\documentclass [openany, oneside,11pt,
a4paper]  {article}

\usepackage{dsfont}
\usepackage{amsmath}
\usepackage{amssymb}	
\usepackage{bbm}
\usepackage{footnote}
\usepackage{nameref} 
\usepackage[a4paper,
scale =0.75]{geometry}
\usepackage[onehalfspacing]{setspace}

\usepackage{amsthm} 
                          	
\newtheorem{counter1}{not to be used as environment} [section]
	\newtheorem{The}[counter1]{Theorem}
	\newtheorem*{The*}{Theorem A}
	\newtheorem*{The**}{Theorem B}
	
	\newtheorem{Lem}[counter1]{Lemma}
	\newtheorem{Kor}[counter1]{Corollary}
	
	\theoremstyle{definition}

	\newtheorem{Bem}[counter1]{Remark}

            \newenvironment{Bew}{\begin{proof}[Proof]}{\end{proof}}

\begin{document}

\title{Block bootstrap for the empirical process of long-range dependent data}
\author{Johannes Tewes\footnote{Fakult\"{a}t f\"{u}r Mathematik, Ruhr-Universit\"{a}t Bochum, 44780 Bochum, Germany, Email address: Johannes.Tewes@rub.de; The authors's work has been supported by the Collaborative Research Center ``Statistical modeling of nonlinear dynamic processes'' (SFB 823) of the German Research Foundation (DFG).} }

\maketitle

\begin{abstract}
We consider long-range dependent data. It is shown that the bootstrapped empirical process of these data converges to a semi-degenerate limit. The random part of this limit is always Gaussian. Thus the bootstrap might fail when the original empirical process accomplishes a noncentral limit theorem.\\\\
\noindent {\sf\textbf{Keywords:}} long-range dependence, bootstrap, empirical process.
\end{abstract}

\section{Introduction}

Efron's \cite{Efr} bootstrap provides a strong nonparametric tool for approximating the distribution of many common statistics. For independent and identically distributed data Bickel and Freedman \cite{BiFr} and Singh \cite{Sin} have shown the asymptotic validity of this procedure. That means the bootstrapped statistics converges to the same limit distribution as the original statistic. The so-called blockwise bootstrap was first considered by K\"{u}nsch \cite{Kuen2} and applies to a large class of weakly dependent random variables. Especially for empirical processes this is of great interest. Let $(X_i)_{i \geq 0}$ be a stationary, weakly dependent time series. Then under some technical assumptions the normalized empirical process $n^{-1/2} \sum_{i=1}^n (1_{\{X_i \leq x \}} - F(x))$ converges to a zero-mean Gaussian process $K(x)$ with covariance kernel
\begin{align}
\begin{aligned}
E[K(x) K(y) ] = F(x \wedge y ) -F(x) F(y) + & \  \sum_{d=1}^{\infty} (P(X_0 \leq x, X_d \leq y) -F(x) F(y)) \\
+& \   \sum_{d=1}^{\infty} (P(X_0 \leq y, X_d \leq x) -F(x) F(y)). 
\end{aligned} \label{KovKernEinleitungWeakDep}
\end{align}
$F(x)$ is the distribution function of $X_i$ and is typically unknown. Even if it is known, (\ref{KovKernEinleitungWeakDep}) is of infinite dimension and cannot be computed. In the case of long-range dependence the situation is different. For several types of long-range dependence (see Dehling and Taqqu \cite{DeTa}, Ho and Hsing \cite{HoHs} and Wu \cite{Wu}) the empirical process converges weakly to $g(x) Z$, where $g$ is a deterministic function and $Z$ a possibly non Gaussian real valued random variable. So the limiting process is not as hard to treat as in the weakly dependent case. In the case of linear processes $g(x)$ is just the probability density and therefore can be estimated. However, in the case nonlinear transformations the function $g(x)$ is not known and hence a resampling method might be of interest. Lahiri \cite{Lah} considered the block bootstrap for the sample mean of long memory processes and showed that it is valid if and only if the non bootstrapped sample mean (properly normalized) converges to a normal limit. It turns out that the bootstrap for the empirical process behaves similar. It converges also to a semi-degenerate limit, but the random part is always normal. Nevertheless, even if the bootstrap technically fails, it has still some statistical applications. The reason is that the deterministic part of the limit, the function $g(x)$, is the same as for the original empirical process. Thus this function can always be estimated using the block bootstrap.

\section{Main results}

Consider the stationary Gaussian process $(X_i)_{i \geq 1}$ with
\begin{align*}
E X_i =0, \ \  EX_i^2 =1 \ \ \text{and} \ \ \rho (k) = E[ X_0 X_k ] = k^{-D} L(k)
\end{align*}
for $0 < D <1$ and a slowly varying function $L$. We will not observe the $X_i$ themselves but a (possibly non-linear) transformation of them, namely $Y_i =G(X_i)$. The  empirical process of these random variables is
\begin{align*}
W_n(x) = \sum_{i=1}^n (1_{\{Y_i \leq x \}} - F(x) ).
\end{align*}
Its asymptotic behavior depends on the so-called Hermite rank, defined by
\begin{align*}
m =  \min \left \{ q >0  \ \vert  \ E [1_{\{ G(X_1) \leq x \} } H_q(X_1)] \not =0 \text{ for some } x \right \}.
\end{align*}
Together with the parameter $D$ it determines the dependence structure of $ \{ 1_{\{ G(X_i) \leq x \} }, \ x \in \mathds R \}_{i \geq 1}$. The correct normalization for the empirical process is
\begin{align*}
d_n \sim n^{H} L^{m/2}(n),
\end{align*}
where $H = 1-mD/2$ is called Hurst exponent. Dehling and Taqqu \cite{DeTa} considered the more complicated sequential empirical process and their result reads as follows. 

\begin{The*} [Dehling, Taqqu] \label{SeqEmpProHypo}
Let the class of functions $\{1_{ \{ G(\cdot) \leq x \}  } - F(t) , \ - \infty < x < \infty\}$ have Hermite rank m and let $0 < D < 1/m$. Then
\begin{align}
 d_n^{-1} W_{\lfloor nt \rfloor } (x) \xrightarrow{\mathcal D}  \frac {J_m(x)} {m !} Z_m(t), \label{SeqKonvHypo}
\end{align}
where the convergence takes place in $D([0,1] \times [- \infty,\infty])$, equipped with the uniform topology. 
\end{The*}

As a direct consequence 
\begin{align}
 d_n^{-1} W_n(x) \xrightarrow{\mathcal D}  \frac {J_m(x)} {m !} Z_m(1) \label{NonSeqKonvHypo}
\end{align}
in the space $D[-\infty, \infty ]$.
$Z_m(1)$ is normalized and standardized and it is Gaussian if and only if $m=1$. $J_m(x)$ is a deterministic function defined by 
\begin{align*}
J_m(x) = E[1_{\{G(X_1) \leq x \}} H_m(X_1)].
\end{align*}
The limit is therefore sometimes called semi-degenerate. $J_m$ depends on the transformation $G$ and to the best of our knowledge there exists no procedure to estimate it. \\
In this paper we will discuss the block bootstrap as possible solution. For a sample $Y_1,\dots, Y_n$ choose a block length $l(n)$ and consider the  $n-l+1$ blocks $I_1, \dots , I_{n-l+1}$, defined by
\begin{align*}
I_j = (Y_{j}, \dots , Y_{j+l-1}) \ \ \ \ j=1, \dots ,n-l+1.
\end{align*}
Then we choose randomly with replacement $p= p(n)$ blocks, so that the bootstrap sample $Y_1^{\ast}, \dots , Y_{pl}^{\ast}$ satisfies
\begin{align*}
 P \left ( (Y_{(j-1)l +1}^{\ast}, \dots , Y_{jl}^{\ast}) = I_i \right ) = \frac{1}{n-l+1} \ \ \text{for} \ \ j = 1, \dots , p, \ \  i =1, \dots ,n-l+1.
 \end{align*}
 The common choice for the number of blocks is $p= \lfloor n /l \rfloor$, however, this is not necessary for the proof. Further denote the blocks of indices by
 \begin{align*}
 B_i = (i, \dots , i+l-1)  \ \ \ \ i=1, \dots ,n-l+1.
\end{align*}
 This procedure is called moving block bootstrap (MBB), see K\"{u}nsch \cite{Kuen2}. In the case of long-range dependence it has been applied to subordinated gaussian processes by Lahiri \cite{Lah} and to linear sequences by Kim and Nordman \cite{KiNo}. Both consider the bootstrap of the sample mean. \\
 In what follows $E^{\ast}$ will denote conditional expectation given the sample $Y_1, \dots , Y_n$. Analogously $P^{\ast}$ denotes conditional probability and $\xrightarrow{\mathcal D}_{\ast}$ weak convergence with respect to $P^{\ast}$.
 
 \begin{The**} [Lahiri] \label{BootsSampleMean}
 Let $l=O(n^{1- \epsilon})$ for some $0 < \epsilon <1$ and $p^{-1}+l^{-1} =o(1)$. Then 
 \begin{align*}
 \frac 1 { p^{1/2} d_l} \sum_{i=1}^{pl} (Y_i^{\ast} -  E^{\ast} Y_i^{\ast} ) \xrightarrow{\mathcal D}_{\ast} \mathcal N (0,\sigma_m^2) \ \ \text{in probability,}
 \end{align*}
 where $\sigma_m = E[G(X_1)H_m(X_1)]/m!$.
 \end{The**}
 
 Two things are remarkable. The first is that the bootstrap destroys somehow the dependence of the random variables, thus a weaker normalization is needed. The second is that the limit is always normal. However, for Hermite ranks larger than $1$ the partial sum of long-range dependent data converges towards a nonnormal limit, see Taqqu \cite{Taq3} and Dobrushin and Major \cite{DoMa}. Hence the bootstrap fails in this case. The sampling window method does not suffer from this issue (see Hall, Jing and Lahiri \cite{HaJiLa}) and has become more popular for statistical inference on long memory time series (see Lahiri and Nordman \cite{LaNo} and Ho et. al. \cite{HoWeWuZh}).
 \\\\
 Now consider the bootstrapped empirical process
 \begin{align*}
  \frac 1 { p^{1/2} d_l} \sum_{i=1}^{pl} (1_{\{ Y_i^{\ast} \leq x \}} - E^{\ast}[1_{\{ Y_i^{\ast} \leq x \}} ]).
  \end{align*}
 For weakly dependent data this was considered by B\"{u}hlmann \cite{Buel}, Naik-Nimbalakar and Rajarshi \cite{NaRa} and Peligrad \cite{Pel}. The main theorem of this paper reads as follows. 
 
 \begin{The} \label{ConvergenceEmpProBoots}
Let the class of functions $\{1_{ \{ G(\cdot) \leq x \}  } - F(t) , \ - \infty < x < \infty\}$ have Hermite rank m and let $0 < D < 1/m$. Let further the block length satisfy $l=O(n^{1- \epsilon})$ for some $0 < \epsilon <1$ and $p^{-1}+l^{-1} =o(1)$. Then 
\begin{align*}
 \frac 1 { p^{1/2} d_l} \sum_{i=1}^{pl} (1_{\{ Y_i^{\ast} \leq x \}} - E^{\ast}[1_{\{ Y_i^{\ast} \leq x \}} ]) \xrightarrow{\mathcal D}_{\ast} \frac {J_m(x)} {m !} Z \ \ \text{in probability},
\end{align*}
where the convergence takes place in $D([- \infty,\infty])$, equipped with the uniform topology. $J_m$ is defined as above and $Z$ is standard normal distributed.  
\end{The}
Similar to the empirical process of LRD data (see (\ref{NonSeqKonvHypo})) the bootstrapped version has a semi-degenerate limit. However, the normalization in Theorem \ref{ConvergenceEmpProBoots} is weaker than in (\ref{NonSeqKonvHypo}) and the random part of the limit is always Gaussian, just as for the bootstrapped sample mean.

\begin{Bem}
The definition of the convergence obtained in Theorem \ref{ConvergenceEmpProBoots} is not straightforward. We say that a random process $Z_n^{\ast}(x)$ converges in probability in distribution if every subsequence $(n_k)_k$ has another subsequence$(n_{k_l})_l$, such that $Z_{n_{k_l}}^{\ast}(x)$ converges almost surely in distribution, see Naik-Nimbalakar and Rajarshi \cite{NaRa}.
\end{Bem}

Comparing the asymptotic distributions in Theorems \ref{SeqEmpProHypo} and \ref{ConvergenceEmpProBoots}, one might conclude that the bootstrap fails if $m >1$. However, the function $J_m(x)$ can still be estimated (up to its sign). \\
Consider $A$ bootstrap iteration and denote by
\begin{align*}
X_{1,a}^{\ast}, \dots , X_{pl,a}^{\ast} \ \ \ \ a \in \{1, \dots , A\}
\end{align*}
the $a$-th bootstrap sample. Denote further the empirical process of the $a$-th sample by $W_{n,a}^{\ast}(x)$. Then our estimator for $J_m(x)$ is given by
\begin{align*}
\hat J_{n,A,m}(x) = m!  \left (\frac1 A \sum_{a=1}^A (W_{n,a}^{\ast}(x))^2 \right )^{1/2}.
\end{align*}

\begin{Kor}
Let the conditions of Theorem \ref{ConvergenceEmpProBoots} hold. Then
\begin{align*}
\lim_{A \to \infty} \lim_{n \to \infty} P \left ( \sup_{x \in \mathds R} \left \lvert \lvert \hat J_{n,A,m} (x) \rvert - \lvert J_m(x) \rvert \right \rvert > \epsilon \right ) = 0, 
\end{align*}
for all $\epsilon >0$.
\end{Kor}

The main part of the proof of Theorem A is a reduction principle and this technique has become popular for empirical processes of LRD data ever sine. Define
\begin{align}
S_n(x) = \frac 1 {d_n} \sum_{i=1}^n \left (1_{\{Y_i \leq x \}} - F(x) - J_m/m!(x) H_m(X_i) \right ). \label{ApproxEmpProRot}
\end{align}
Dehling and Taqqu \cite{DeTa} have shown that $S_n$ converges uniformly and in probability towards zero. It is our aim to proof Theorem \ref{ConvergenceEmpProBoots} in a similar way. To this end consider the bootstrapped version of (\ref{ApproxEmpProRot})
\begin{align}
S_{n,l}^{\ast}(x) = \frac 1 {d_l p^{1/2}} \sum_{i=1}^{pl} \left  (1_{\{Y^{\ast}_i \leq x \}} - \tilde F_{n,l}(x) - J_m(x)/m! \left ( H_m(X_i^{\ast}) - \tilde \mu_{n,l}(H_m) \right ) \right ),
\end{align}
where 
\begin{align}
\tilde \mu_{n,l}(H_m) = l^{-1}  E^{\ast} \left [ \sum_{j \in B_1} H_m(X_j^{\ast})  \right ] \ \ \ \ \text{and} \ \ \ \  \tilde F_{n,l}(x) = l^{-1}  E^{\ast} \left [ \sum_{j \in B_1} 1_{\{Y_j^{\ast} \leq x \}} \right  ]. \label{DefinitionAltMeanAltECDF}
\end{align}

\begin{Lem}[Bootstrap uniform weak reduction principle] \label{WeakRedPrinBoots}
Let the conditions of Theorem \ref{ConvergenceEmpProBoots} hold. Then
\begin{align*}
P^{\ast} \left ( \sup_{ - \infty \leq x \leq \infty} \lvert S_{n,l}^{\ast}(x) \rvert > \epsilon \right ) \rightarrow 0  \ \ \text{in probability} 
\end{align*}
for all $\epsilon >0$ and $n \to \infty$.
\end{Lem}  

\section{Preliminary results}

Introduce some notation:
\begin{align*}
& S_n(x,y) = S_n(y) - S_n(x),  \ \ \ \  F(x,y) = F(y) -F(x) \\
& \tilde F_{n,l}(x,y) = \tilde F_{n,l}(y) - \tilde F_{n,l}(x),  \ \ \ \ J_m(x,y) = J_m(y) -J_m(x).
\end{align*}

\begin{Lem} [Dehling, Taqqu] \label{PreRot}
There exists constants $\gamma >0$ and $C >0$ such that for all $n \in \mathds N$
\begin{align*}
E \lvert S_n(x,y) \rvert ^2 \leq C n^{-\gamma} (F(y) -F(x)).
\end{align*}
\end{Lem}

The next result is Lemma 3.1. of Lahiri \cite{Lah}.

\begin{Lem} [Lahiri] \label{PreBlau}
Define $\tilde \mu_{n.l}(H_m)$ as in (\ref{DefinitionAltMeanAltECDF}). 
If the conditions of Theorem \ref{ConvergenceEmpProBoots} hold
\begin{align*} 
\text{(i)} \ \tilde \mu_{n,l}(H_m)  =o_P(d_l/l) \ \ \ \ \text{and} \ \ \ \ \text{(ii)} \ E [\mu_{n,l}(H_m)]^2 \leq C d_n^2/n^2.
\end{align*}
\end{Lem}

The next lemma extends the previous one to indicator functions.

\begin{Lem} \label{PreGelb}
Define $\tilde F_{n,l}(x)$ as in (\ref{DefinitionAltMeanAltECDF}). If the conditions of Theorem \ref{ConvergenceEmpProBoots} hold
\begin{align*}
E \left ( F(x,y) - \tilde F_{n,l}(x,y) \right )^2 \leq C d_n^2/n^2 F(x,y).
\end{align*}
\end{Lem}

\begin{Bew}
Since the Hermite rank equals $m$ we obtain the following expansion
\begin{align*}
1_{\{x < Y_j \leq y\}} - F(x,y) = \sum_{q=m}^{\infty} J_q(x,y)/q! H_q(X_i).
\end{align*}
By definition of $\tilde F_{n,l}(x)$ we have
\begin{align*}
F(x) - \tilde F_{n,l}(x) = & \ F(x) -\frac 1 l \frac 1 {(n-l+1)} \sum_{j=1}^n a_{n,j}1_{\{Y_j\leq x \}} \\
= & \  \frac 1 l \frac 1 {(n-l+1)} \sum_{j=1}^n a_{n,j}(F(x) - 1_{\{Y_j\leq x \}} ),
\end{align*}
where
\begin{align*}
a_{n,j} = \begin{cases} 
                             j,  & \text{if } j < l  ,\\
                          l, & \text{if } l \leq j \leq n-l +1, \\
                          n -j +1 & \text{if } j > n-l+1.
                                                      \end{cases}
\end{align*}
Note that $a_{n,j} \leq l$ for all $j$. By orthogonality of the $H_q(X_i)$,
\begin{align*}
\sum_{q=m}^{\infty} J_q^2(x,y)/q! \leq F(x,y)
\end{align*}
and moreover
\begin{align*}
E \left ( F(x,y) - \tilde F_{n,l}(x,y) \right )^2 = & \ \frac 1 {l^2} \frac 1 {(n-l+1)^2} \sum_{q=m}^{\infty} \frac{J_q^2(x,y)} {q!} \frac 1 {q!} \sum_{i,j \leq n} a_{n,i}a_{n,j} E[H_q(X_i)H_q(X_j)] \\
\leq & \ \frac 1 {(n-l+1)^2} F(x,y) \sum_{i,j \leq n} \lvert r(i-j) \rvert^m.
\end{align*}
The conclusion follows because $d_n^2 \sim \sum_{i,j \leq n} \lvert r(i-j) \rvert^m$.
\end{Bew}

\section{Proof of the main result}

\begin{proof} [Proof of Lemma \ref{WeakRedPrinBoots}.]
We will proof the result by using exactly the same chaining points as in Dehling and Taqqu \cite{DeTa}. Define
\begin{align*}
\Lambda (x) := F (x)+ \int_{\{G(s)\leq x\}} \frac {\lvert H_{m}(s) \rvert} {m!} \phi(s) \ ds.
\end{align*} 
The function $\Lambda$ is monotone, $\Lambda(-\infty) =0$, $\Lambda(+\infty) < \infty$ and $\max\{F(x,y), J_m(x,y)/m!\} \leq \Lambda(y) - \Lambda(x)$. \\
Define for $k=0,1, \dots, K$ refining partitions of $\mathds R$,
\begin{align*}
 - \infty =x_i(k) \leq x_1(k) \leq \dots \leq x_{2^k}(k) = \infty,
 \end{align*}
 by 
 \begin{align*}
 x_i(k) = \inf \{ x \in \mathds R \ \vert \ \Lambda(x) \geq \Lambda(+ \infty) i2^{-k} \},  \ \ i =0,1, \dots , 2^k-1.
 \end{align*}
 $K$ will be chosen later. Then we have
\begin{align*}
\Lambda(x_i(k)-) - \Lambda(x_{i-1}(k)) \leq \Lambda(+\infty)2^{-k}. 
\end{align*}
Based on these partitions we can define chaining points $i_k(x)$ by
\begin{align*}
x_{i_k(x)}(k) \leq x < x_{i_k(x)+1}(k),
\end{align*}
 for each $x$ and each $k \in \{0,1, \dots , K\}$, see Dehling and Taqqu \cite{DeTa}. In this way each point $x$ is linked to $-\infty$, in detail  
\begin{align*}
- \infty = x_{i_0(x)}(0) \leq x_{i_1(x)}(1) \leq \dots \leq x_{i_K(x)}(K) \leq x.
\end{align*}
We have
\begin{align}
 \nonumber S_{n,l}^{\ast} (x) = & \ S_{n,l}^{\ast} (x_{i_0(x)}(0),x_{i_1(x)}(1)) \\
 \nonumber & \ + S_{n,l}^{\ast} (x_{i_1(x)}(1),x_{i_2(x)}(2)) \\
& \ + \cdots \label{ErsteZerlegungViolett} \\
 \nonumber  & \ + S_{n,l}^{\ast} (x_{i_{K-1}(x)}(K-1),x_{i_K(x)}(K)) \\  
 \nonumber& \ + S_{n,l}^{\ast} (x_{i_K(x)}(K),x),
\end{align}
where $S_{n,l}^{\ast} (x,y) = S_{n,l}^{\ast} (y) - S_{n,l}^{\ast} (x)$. \\
Let us first consider the last term of (\ref{ErsteZerlegungViolett}). We get
\begin{align*}
&  \ \lvert S_{n,l}^{\ast} (x_{i_K(x)}(K),x) \rvert \\
 = & \ \Bigg \lvert d_l^{-1}p^{-1/2}\sum_{j =1}^{pl} \bigg ( 1_{\{ x_{i_K(x)}(K) < Y_j^{\ast} \leq x \} } - \tilde F_{n,l}(x_{i_K(x)}(K) ,x)  \\
& \ - \frac 1 {m!} J_{m}(x_{i_K(x)}(K),x) (H_{m}(X_j^{\ast})-\tilde \mu_{n,l}(H_m)) \bigg )\Bigg \rvert \\
\leq & \ d_l^{-1}p^{-1/2} \sum_{j =1}^{pl} \left ( 1_{\{ x_{i_K(x)}(K) < Y_j^{\ast} \leq x \} } + \tilde F_{n,l}(x_{i_K(x)}(K) ,x) \right) \\
 & \ + \left \lvert \frac 1 {(m)!} J_{m}(x_{i_K(x)}(K),x) d_l^{-1}p^{-1/2} \sum_{j =1}^{pl} (H_{m}(X_j^{\ast})-\tilde \mu_{n,l}(H_m))  \right \rvert \\
\leq & \ \left \lvert S_{n,l}^{\ast} (x_{i_{K}(x)}(K),x_{i_K(x)+1}(K)-) \right \rvert \\
 & \ + 2pl d_l^{-1}p^{-1/2} \tilde F_{n,l}(x_{i_{K}(x)}(K),x_{i_K(x)+1}(K)-)  \\
& \ + 2 \Lambda(+\infty) 2^{-K} d_l^{-1}p^{-1/2} \left \lvert \sum_{j =1}^{pl}( H_{m}(X_j^{\ast})-\tilde \mu_{n,l}(H_m) )\right \rvert \\
\leq & \ \left \lvert S_{n,l}^{\ast} (l;x_{i_{K}(x)}(K),x_{i_K(x)+1}(K)-) \right \rvert \\
 & \ + 2pl d_l^{-1}p^{-1/2} \left ( \tilde F_{n,l}(x_{i_{K}(x)}(K),x_{i_K(x)+1}(K)-)  - F(x_{i_{K}(x)}(K),x_{i_K(x)+1}(K)-) \right ) \\
 & \ + 2pl d_l^{-1}p^{-1/2}  F(x_{i_{K}(x)}(K),x_{i_K(x)+1}(K)-)  \\
 & \ +2 \Lambda(+\infty) 2^{-K} d_l^{-1}p^{-1/2} \left \lvert \sum_{j =1}^{pl} ( H_{m}(X_j^{\ast})-\tilde \mu_{n,l}(H_m)) \right \rvert.
\end{align*}
Note that $\sum_{k=0}^{\infty} \epsilon / (k+3)^2 \leq \epsilon /2$.  Making further use of the estimate above and the decomposition (\ref{ErsteZerlegungViolett}) we get
\begin{align}
 \nonumber & \ P^{\ast} \left ( \sup_{x} \lvert S_{n,l}^{\ast}(x) \rvert > \epsilon \right ) \\
 \nonumber \leq & \ P^{\ast} \left ( \sup_{x} \lvert S_{n,l}^{\ast}(x) \rvert > \epsilon \sum_{k=0}^K(k+3)^{-2} +\epsilon/2 \right ) \\
 \nonumber \leq & \ P^{\ast} \left ( \max_{x} \lvert S_{n,l}^{\ast}(x_{i_0(x)}(0),x_{i_1(x)}(1)) \rvert > \epsilon /9 \right ) \\
 \nonumber   & \ + P^{\ast} \left ( \max_{x} \lvert S_{n,l}^{\ast}(x_{i_1(x)}(1),x_{i_2(x)}(2)) \rvert > \epsilon /16 \right ) \\
 & \ + \cdots \label{DieLetzteWahrscheinlihckeit} \\
 \nonumber  & \ + P^{\ast} \left ( \max_{x} \lvert S_{n,l}^{\ast}(x_{i_{K}(x)}(K),x_{i_K(x)+1}(K)-) \rvert > \epsilon /(K+3)^2 \right )\\
 \nonumber  & \ + P^{\ast} \left ( \max_{x} 2 pl d_l^{-1}p^{-1/2} \left \lvert \tilde F_{n,l}(x_{i_{K}(x)}(K),x_{i_K(x)+1}(K)-)  - F(x_{i_{K}(x)}(K),x_{i_K(x)+1}(K)-) \right \rvert  > \epsilon/(K+4)^2 \right ) \\ 
 \nonumber  & \ + P^{\ast} \left ( 2 \Lambda(+\infty) 2^{-K} d_l^{-1}p^{-1/2} \left \lvert \sum_{j =1}^{pl} ( H_{m}(X_j^{\ast})-  E^{\ast} [H_{m}(X_j^{\ast})])  \right \rvert > (\epsilon/2) - 2 \Lambda(+\infty) pl d_{l}^{-1} p^{-1/2} 2^{-K} \right ).
\end{align}
By the Markov inequality we get
\begin{align}
\nonumber & \ P^{\ast}  \left ( \max_{x} \lvert S_{n,l}^{\ast}(x_{i_k(x)}(k),x_{i_{k+1}(x)}(k+1)) \rvert > \epsilon /(k+3)^2 \right ) \\
\nonumber \leq & \ \sum_{i=0}^{2^{k+1}-1} P^{\ast} \left ( S_{n,l}^{\ast}(x_i(k+1),x_{i+1}(k+1))  > \epsilon /(k+3)^2 \right ) \\
\leq & \ \sum_{i=0}^{2^{k+1}-1} E^{\ast} \left [ S_{n,l}^{\ast} (x_i(k+1),x_{i+1}(k+1)) \right ]^2 \frac {(k+3)^4}{\epsilon^2}. \label{EstimateMarkovInequality}
\end{align}
By construction of the bootstrap sample we get
\begin{align*}
& \ E^{\ast} [ S_{n,l}^{\ast} (x) ]^2 \\
= & \ \frac 1 {d_l^{2}p} E^{\ast} \left [ \sum_{j=1}^{kl} (1_{\{Y_j^{\ast} \leq x \}}  - \tilde F_{n,l}(x) - J_m(x)/m! (H_m(X_j^{\ast}) - \tilde \mu_{n,l}(H_m)) \right ]^2 \\
= & \ \frac 1 {d_l^{2}} E^{\ast} \left [ \sum_{j \in B_1} (1_{\{Y_j^{\ast} \leq x \}}  - \tilde F_{n,l}(x) - J_m(x)/m! (H_m(X_j^{\ast}) - \tilde \mu_{n,l}(H_m)) ) \right ]^2 \\
= & \ \frac 1 {d_l^{2}} \frac 1 {(n-l+1)} \sum_{i=1}^{n-l+1} \left ( \sum_{j \in B_i} (1_{\{Y_j \leq x \}}  - \tilde F_{n,l}(x) - J_m(x)/m! (H_m(X_j) - \tilde \mu_{n,l}(H_m)) ) \right )^2 \\
\leq & \ \frac 1 {d_l^{2}} \frac 1 {(n-l+1)} C \sum_{i=1}^{n-l+1} \left ( \sum_{j \in B_i} (1_{\{Y_j \leq x \}}  - F(x) - J_m(x)/m! H_m(X_j) )\right )^2 \\
& \ + \frac 1 {d_l^{2}} C l^2 \left ( F(x) - \tilde F_{n,l}(x) \right )^2 \\
& \ + \frac 1 {d_l^{2}} C J_m^2(x)/(m!)^2 l^2 \left (\tilde \mu_{n,l}(H_m)\right )^2 \\
= & \  \frac 1 {(n-l+1)} C \sum_{i=1}^{n-l+1} S_{l,i}^2(x) \\
& \ + \frac 1 {d_l^{2}} C l^2 \left ( F(x) - \tilde F_{n,l}(x) \right )^2 \\
& \ + \frac 1 {d_l^{2}} C J_m^2(x)/(m!)^2 l^2 \left ( \tilde \mu_{n,l}(H_m)\right )^2,
\end{align*}
where 
\begin{align*}
S_{l,i}(x) = \frac 1 {d_l} \sum_{j \in B_i} (1_{\{Y_j \leq x \}}  - F(x) - J_m(x)/m! H_m(X_j) ).
\end{align*}
Consequently 
\begin{align}
\nonumber & \ E^{\ast} [ S_{n,l}^{\ast} (x,y) ]^2 \\
\nonumber \leq & \  \frac 1 {(n-l+1)} C \sum_{i=1}^{n-l+1} S_{l,i}^2(x,y) \\
& \ + \frac 1 {d_l^{2}} C l^2 \left ( F(x,y) - \tilde F_{n,l}(x,y) \right )^2 \label{EstimateBootstrapProcedure} \\
\nonumber & \ + \frac 1 {d_l^{2}} C J_m^2(x,y)/(m!)^2 l^2 \left ( \tilde \mu_{n,l}(H_m)\right )^2,
\end{align}

It is our goal to show that $E[P^{\ast}(\sup_{x \in \mathds R} \lvert S_{n,l}^{\ast} (x) \rvert > \epsilon)] \rightarrow 0$ as $n \to \infty$. To this end we take the expectation of every summand of the right-hand side of (\ref{DieLetzteWahrscheinlihckeit}). Making then successive use of the estimates (\ref{EstimateMarkovInequality}) and (\ref{EstimateBootstrapProcedure}) we obtain
\begin{align*}
& \ E \left [  P^{\ast}  \left ( \max_{x} \lvert S_{n,l}^{\ast}(x_{i_k(x)}(k),x_{i_{k+1}(x)}(k+1)) \rvert > \epsilon /(k+3)^2 \right )  \right ] \\
= & \ C \sum_{i=0}^{2^{k+1}-1} E [S_{l}^2(x_i(k+1),x_{i+1}(k+1))] \frac {(k+3)^4}{\epsilon^2} \\
& \ + C \sum_{i=0}^{2^{k+1}-1} \frac {l^2} {d_l^2} E \left ( F(x_i(k+1),x_{i+1}(k+1)) - \tilde F_{n,l}(x_i(k+1),x_{i+1}(k+1)) \right )^2 \frac {(k+3)^4}{\epsilon^2} \\ 
& \ + C \sum_{i=0}^{2^{k+1}-1} \frac {J_m^2(x_i(k+1),x_{i+1}(k+1))} {(m!)^2} \frac 1 {d_l^2} l^2 E \left ( \tilde \mu_{n,l}(H_m) \right )^2 \frac {(k+3)^4}{\epsilon^2} \\
\leq & \ C \sum_{i=0}^{2^{k+1}-1} l^{-\gamma} F(x_i(k+1),x_{i+1}(k+1)) \frac {(k+3)^4}{\epsilon^2} \\
& \ + C \sum_{i=0}^{2^{k+1}-1} \frac {l^2} {d_l^2} \frac {d_n^2}{n^2} F(x_i(k+1),x_{i+1}(k+1)) \frac {(k+3)^4}{\epsilon^2} \\
& \ + C \sum_{i=0}^{2^{k+1}-1} \Lambda(x_i(k+1),x_{i+1}(k+1))^2 \frac 1 {d_l^2} l^2 E \left ( \tilde \mu_{n,l}(H_m) \right )^2 \frac {(k+3)^4}{\epsilon^2}.
\end{align*}
We have also used Lemma \ref{PreGelb} and 
\begin{align*}
E \lvert S_{l,i}(y)-S_{l,i}(x)  \rvert ^2 \leq C l^{-\gamma} (F(y) -F(x))
\end{align*}
which is implied by Lemma \ref{PreRot}. Note that
$(l/n)^2(d_n/d_l)^2 \leq C l^{\lambda}$ for some $\lambda >0$ and $\Lambda(x_i(k+1),x_{i+1}(k+1))^2 \leq C 2^{-2(k+1)}$. Thus setting $\eta = \min \{ \gamma, \lambda \}$ yields
\begin{align*}
& \ E \left [  P^{\ast}  \left ( \max_{x} \lvert S_{n,l}^{\ast}(x_{i_k(x)}(k),x_{i_{k+1}(x)}(k+1)) \rvert > \epsilon /(k+3)^2 \right )  \right ] \\
= & \ C \left ( l^{-\eta} (k+3)^4 \epsilon^{-2} +  2^{-(k+1)} l^2/d_l^2 E[\tilde \mu_{n,l}(H_m)]^2 \right ).
\end{align*}
In the same way we get
\begin{align*}
E \left [  P^{\ast} \left ( \max_{x} \lvert S_{n,l}^{\ast}(x_{i_{K}(x)}(K),x_{i_K(x)+1}(K)-) \rvert > \epsilon /(K+3)^2 \right ) \right ] \\
\leq C l^{-\eta} (K+3)^4 \epsilon^{-2} + C 2^{-K} l^2/d_l^2 E[\tilde \mu_{n,l}(H_m)]^2
\end{align*}
and 
\begin{align*}
& \ E[ P^{\ast} \left ( \max_{x} 2 pl d_l^{-1}p^{-1/2} \left \lvert \tilde F_{n,l}(x_{i_{K}(x)}(K),x_{i_K(x)+1}(K)-)  - F(x_{i_{K}(x)}(K),x_{i_K(x)+1}(K)-) \right \rvert  > \epsilon/(K+4)^2 \right ) \\
\leq & \ \sum_{i=0}^{2^{K-1}} 2 \frac{pl}{d_l p^{1/2}} \frac {(K+4)^4}{\epsilon^2} E  \left ( F(x_i(K),x_{i+1}(K)-) - \tilde F_{n,l}(x_i(K),x_{i+1}(K)-) \right )^2 \\
\leq & \ C l^{-\eta} \frac {(K+4)^4} {\epsilon^{2}}.
\end{align*}
Choose now 
\begin{align*}
K = \left [ \log_2 \left ( \frac {8 \Lambda(+ \infty)}{\epsilon} l d_l^{-1} p^{1/2} \right ) \right ] +1,
\end{align*}
hence $ 2 \Lambda(+\infty)pl  d_{l}^{-1} p^{-1/2} 2^{-K} \leq \epsilon/4$. It remains to treat the last probability in (\ref{DieLetzteWahrscheinlihckeit}). By our choice of $K$ it can be bounded by
\begin{align}
\nonumber & \ P^{\ast} \left ( d_l^{-1}p^{-1/2} \left \lvert \sum_{j =1}^{pl} ( H_{m}(X_j^{\ast})-E^{\ast} [H_{m}(X_j^{\ast})])  \right \rvert > \frac {\epsilon}{4} \frac {2^{K-1}}{\Lambda(+\infty)} \right ) \\
\leq & \ d_l^{-2} p^{-1} E^{\ast} \left [ \sum_{j =1}^{pl} ( H_{m}(X_j^{\ast})-E^{\ast} [H_{m}(X_j^{\ast})]) \right ]^2 \frac {16}{\epsilon^2} \Lambda(+\infty)^2 2^{-2K+2}. \label{LetzteAbschGrau}
\end{align}
By the proof of Theorem B (see Lahiri \cite{Lah}) we get
\begin{align*}
d_l^{-2} p^{-1} E \left (E^{\ast} \left [ \sum_{j =1}^{pl} ( H_{m}(X_j^{\ast})-E^{\ast} [H_{m}(X_j^{\ast})]) \right ]^2 \right ) \leq C. 
\end{align*}
Taking expectation in (\ref{LetzteAbschGrau}) therefore yields
\begin{align*}
E \left [ P^{\ast} \left ( d_l^{-1}p^{-1/2} \left \lvert \sum_{j =1}^{pl} ( H_{m}(X_j^{\ast})-E^{\ast} [H_{m}(X_j^{\ast})])  \right \rvert > \frac {\epsilon}{4} \frac {2^{K-1}}{\Lambda(+\infty)} \right ) \right ] \leq & \ C \frac {16}{\epsilon^2} \Lambda(+\infty)^2 2^{-2K+2} \\
\leq & \ C l^{-2}p^{-1}d_l^2.
\end{align*}
We have now found estimates for the expectation of all summands of (\ref{DieLetzteWahrscheinlihckeit}). Combining these estimates we find
\begin{align*}
E\left [ P^{\ast} ( \sup_{ x} \lvert S_{n,l}^{\ast} (x) \rvert > \epsilon ) \right] \leq & \ C l^{-\eta} \epsilon^{-2} \sum_{k=0}^{K+1} (k+3)^4 + l^2 d_l^{-2} E[\tilde \mu_{n,l}(H_m)]^2 \sum_{k=0}^K 2^{-(k+1)} + C l^{-2}p^{-1}d_l^2 \\
\leq & \ C l^{-\eta} \epsilon^{-2} (K+4)^5 + C l^2 d_l^{-2} E[\tilde \mu_{n,l}(H_m)]^2 + Cl^{-2} p ^{-1}d_l^2 \\
\leq & \ C  l^{-\eta} \epsilon^{-2} (K+4)^5 + C l^{-\eta} + C l^{2H-2}.
\end{align*}
In the last line we have used $l^2 d_l^{-2} E[\tilde \mu_{n,l}(H_m)]^2 \leq Cl^{-\lambda} \leq C l^{-\eta }$ (see Lemma \ref{PreBlau} (ii)) and $l^{-2}p^{-1}d_l^2 \leq l^{2H-2}p^{-1}L^{m/2}(l) \leq l^{-\alpha} $ for $\alpha >0$. \\
The definition of $K$ yields
\begin{align*}
(K+4)^5 \leq C \left ( \lvert \log (\epsilon^{-1}) \rvert^5 + \lvert \log(pl) \rvert^5 \right ) \leq C \epsilon^{-1} l^{\delta},
\end{align*}
 for any $\delta >0$ and a constant $C$, depending on $\delta$. Choose $\delta = \eta/2$ and $ \rho =\min \{\eta - \delta, \alpha \}$, then
 \begin{align*}
 E\left [ P^{\ast} ( \sup_{ x} \lvert S_{n,l}^{\ast} (x) \rvert > \epsilon ) \right] \leq C  l^{-\rho}(\epsilon^{-3} +1).
 \end{align*}

\end{proof}

\begin{proof}[Proof of Theorem \ref{ConvergenceEmpProBoots}.]
By Theorem B, which is the main result of Lahiri \cite{Lah}, we have
\begin{align*}
\frac 1 {d_l p^{1/2}} \sum_{i=1}^{pl} (H_m(X_i^{\ast}) - E^{\ast} [H_m(X_i^{\ast})]) \xrightarrow{\mathcal D}_{\ast} Z  \ \ \text{in probability,}
\end{align*}
where $Z$ is standard normal distributed. By the boundedness of $J_m(x)$ we get by the continuous mapping theorem
\begin{align*}
\frac 1 {d_l p^{1/2}} \frac {J_m(x)} {m!} \sum_{i=1}^{pl} (H_m(X_i^{\ast}) - E^{\ast} [H_m(X_i^{\ast})]) \xrightarrow{\mathcal D}_{\ast}  \frac {J_m(x)} {m!} Z  \ \ \text{in probability,}
\end{align*}
where the weak convergence takes place in $D[-\infty,\infty]$, equipped with the uniform topology. Together with the reduction principle (Lemma \ref{WeakRedPrinBoots}) this finishes the proof.

\end{proof}

\bibliography{Lit}
 
  \end{document}